\documentclass[final,leqno]{siamltex}

\usepackage{amsmath,amssymb,exscale}

\newtheorem{remark}[theorem]{{\it Remark\/}}

\renewcommand{\thefootnote}{\fnsymbol{footnote}}

\title{Finite Propagation Speed of Waves in Anisotropic Viscoelastic Media
}

\author{Joyce R. McLaughlin\thanks{Department of Mathematical Sciences, Rensselaer Polytechnic Institute, Troy, NY 12180 (mclauj@rpi.edu). The first author was partially supported by ONR grant N00014-13-1-0388.}
\and Jeong-Rock Yoon\thanks{Department of Mathematical Sciences, Clemson University, Clemson, SC 29634-0975 (jryoon@\break clemson.edu).}}

\begin{document}

\maketitle

\setcounter{page}{1}

\renewcommand{\thefootnote}{\arabic{footnote}}

\begin{abstract}
Finite propagation speed properties in mathematical elastic and viscoelastic models are fundamental in many applications where the data exhibits propagating fronts.  We note particularly that this property is observed in biomechanical imaging of tissue, in particular in the supersonic imaging experiment, and also in geophysics and ocean acoustics.  With these applications in mind, noting that there are many other applications as well, we present finite propagation speed results for very general integro-differential, anisotropic, viscoelastic linear models, which are not necessarily of convolution type. We start with work density, define work density decomposition and we achieve our results utilizing energy arguments.  One of the advantages of our presented method, instead of using plane wave arguments, is that there is no need to make the homogeneous medium assumption to obtain the finite propagation speed results.
\end{abstract}




\pagestyle{myheadings}
\thispagestyle{plain}
\markboth{J. R. MCLAUGHLIN AND J.-R. YOON}{FINITE PROPAGATION SPEED IN VISCOELASTIC MEDIA}

\section{Introduction}
Finite propagation speed is a fundamental property in many applications pertaining to elastic media as that is a phenomena observed in the data.  In addition, viscoelasticity and anisotropy are observed and so the mathematical models need to include those properties as well.  We note particularly in biomechanical imaging of tissue, e.g., in the experiment, supersonic imaging \cite{Bercoff}, a sequence of pushes created by focused ultrasound creates a wave with a wave front, indicating finite propagation speed, with time traces exhibiting viscoelastic behavior and where the experiment can be performed in isotropic tissue (breast or prostate) or anisotropic tissue (muscle). Furthermore, the pushes themselves can be created precisely because of the viscoelastic behavior in tissue \cite{S}.    Anisotropy, viscoelasticity and finite propagation speed are also observed in geophysics and in shallow water seabeds.  In this paper we consider linear but, at the same time, very general viscoelastic, anisotropic, inhomogeneous systems, and present a time domain energy method that can be utilized to establish the finite propagation speed property. Our model is an integro-differential equation system as it is this class of models (usually referred to as viscoelastic models) that we have found, in our work, are dissipative, exhibit the finite propagation speed property and provide properties that are consistent with applications of interest to us.  The fundamental expression that we use as our starting point for establishing finite propagation speed is related to a work density decomposition.

In the homogeneous case, a frequently used method is to establish finite propagation speed by taking the time Fourier/Laplace transform, utilizing plane waves, and establishing that at each frequency, the frequency dependent wave speeds observed in the plane wave exponent, are uniformly bounded. See, e.g., \cite{SH}.

Here we move away from the frequency domain approach and develop a space-time domain energy based method that naturally allows the assumption of an inhomogeneous medium.  We establish that: (1) if the solution to our model equations, together with the solution's time derivative, are zero at the base of a certain space-time cone, then the solution is identically zero within that cone thereby establishing finite propagation speed; (2) a finite sum superposition principle where the finite propagation speed property holds for individually defined stress/strain relations implies the finite propagation speed property holds for the sum of those individual stress/strain relations; and (3) a continuous superposition principle, utilizing an infinite integral, that enables us to apply our energy method to the case where the stress/strain relation is derived from the fractional derivative Zener, see \cite{M}, model; this last result applies also to the fractional derivative Maxwell model but not to the fractional derivative Voigt and Newton models where the finite propagation speed property does not hold.

Prior to this work, we developed an energy based argument that establishes the finite propagation speed property for:  (1) the inhomogeneous, isotropic linear acoustic equation, \cite{MY}; (2) the inhomogeneous linear elastic, isotropic system \cite{MY}; and (3) the inhomogeneous, isotropic, viscoelastic generalized  Linear Solid Model \cite{MTY} where we also established smoothness results for the solutions of this model.

In this paper we advance this argument to our much more general viscoelastic anisotropic model starting with a work density formulation and establishing an inequality related to the dissipative property of the system that, when satisfied, establishes finite propagation speed for our general model. We present examples that include, e.g.: (1) an inhomogeneous, isotropic aging model; (2) the generalized Linear Solid Model where we show finite propagation speed using our superposition result; and (3) the fractional derivative Zener or Maxwell models where we utilize our continuous superposition principle to establish finite propagation speed.

The paper is organized as follows.  In Section 2 we present the background. In Section 3 we give our basic energy lemmas.  In Section 4 we present our main finite propagation result.  In Section 5 we advance our finite propagation speed result to a finite sum of anisotropic, viscoelastic models which individually satisfy our criteria and also advance our method to apply to a continuum of anisotropic, viscoelastic models.  In Section 6 we present examples where our theory is applied.

\section{Background}\label{sec:Background}
The goal of this paper is to present a method to establish finite propagation speed for a {\it{general}} linear viscoelastic anisotropic mathematical model. The method is based on defining a work density decomposition which contains the sum of the kinetic energy and a candidate for strain energy and an associated dissipative energy term. We establish that if the initial velocity and initial displacement are zero in a ball in $\mathbb{R}^d$, $d\ge 2$, then the sum of the kinetic energy and the strain energy is zero in a space-time cone implying that the displacement is identically equal to zero in the same cone.  Note that in the purely elastic case the strain energy is inarguably well defined but in the viscoelastic case it is not.

In this section we present a finite propagation speed result, similar to that established in \cite{MTY} for an isotropic generalized Linear Solid Model, where here the isotropic generalized Linear Solid Model with exponential convolution type kernels is slightly further generalized.  Also here we say explicitly the choice for strain energy density that enables the finite propagation speed result.  Furthermore, in keeping with the rest of the paper we assume that we have enough smoothness for all derivatives and integrals to exist. For a discussion of possible choices for the solution and parameter spaces, see \cite {MTY}.

Our generalized Linear Solid Model is
\begin{equation}\label{eq:VI}
\begin{aligned}
\rho(x)\vec{u}_{tt}=&\nabla\left[\lambda^0(x) \nabla\cdot \vec{u} +\sum^{M}_{i=1}\int_0^t\lambda^i(x) e^{-\frac{t-s}{\gamma^i (x)}}\nabla\cdot\vec{u}_{s}(x,s)\,ds\right]\\
 &+2\nabla\cdot \left[\mu^0(x) \epsilon(x,t)+\sum^{N}_{i=1} \int_0^t\mu^{i}(x) e^{-\frac{t-s}{\tau^i(x)}}\epsilon_s(x,s)\,ds\right],
\end{aligned}
\end{equation}
where the mass density $\rho(x)$, the Lam$\acute{\mbox{e}}$ parameter $\lambda^i(x)$, the shear modulus $\mu^i(x)$,
the relaxation times $\gamma^i(x)$ and $\tau^i(x)$ are all assumed to be positive. Here $\vec{u}$ is the displacement,
$\epsilon=\frac 1 2 (\nabla\vec u+(\nabla\vec u)^T)$ is the strain, and subscripts denote time derivatives.
The space-time cone is defined as follows: For any open ball
$B_R(x_0)\subset \mathbb{R}^d$ and any given $c>0$,
$$
\forall s\in (0,R/c),\quad \Lambda(s):=\bigcup_{0<\tau<s} C_\tau
\quad\mbox{where}\quad C_\tau:=B_{R-c \tau}(x_0)\times\{t=\tau\},
$$
where $B_{R-c\tau}(x_0)$ is the ball centered at $x_{0} \in \mathbb{R}^{d}$ of radius
$R-c\tau>0$. Also $\partial \Lambda(s)=C_s\cup C_0\cup L$ with $L=\bigcup_{0<\tau<s} \partial C_\tau$.
The kinetic energy density is $e_{K}(x,t):=\frac{\rho(x)}{2}|\vec u_t(x,t)|^2$, and the strain energy density is \emph{defined} as
$$
\begin{aligned}
e_{S}(x,t)=&\frac{\lambda^{0}(x)}{2}|\nabla\cdot\vec{u}(x,t)|^2+\frac{\lambda^i(x)} 2 \sum^{M}_{i=1}\left|\int^{t}_{0} e^{-\frac{t-s}{\gamma^i(x)}} \nabla\cdot\vec{u}_{s}(x,s)\,ds\right|^{2}\\
&+\mu^{0}(x)|\epsilon(x,t)|^{2}+\sum^{N}_{i=1} \mu^i(x) \left|\int^{t}_{0} e^{-\frac{t-s}{\tau^i(x)}}\epsilon(x,s)\,ds\right|^{2}
\end{aligned}
$$
yielding the total stored energy, $e(s)=\int_{C_s} \left\{e_{K} (x,s) + e_{S}(x,s)\right\}\,dx$.
Our choice of strain energy density is \emph{natural} given the definition of that in the purely elastic isotropic case and this choice enables us to establish the following:

\begin{theorem}
Let $\vec{u}$ satisfy equation \eqref{eq:VI} with $\vec{u}(x,0)=\vec{u}_t(x,0)=0$, $\forall x\in B_{R} (x_{0})$.
Then $\vec{u} \equiv 0$ in the space-time cone, $\Lambda (R/c)$, with
$$
c=\sup_{x\in \bar B_{R} (x_0)} \sqrt{\frac{\sum_{i=0}^M\lambda^i(x)+2\sum_{i=0}^{N}\mu^i(x)}{\rho(x)}}.
$$
\end{theorem}
\begin{proof}
Multiplying equation (\ref{eq:VI}) by $\vec{u} _{t} $ integrating over $\Lambda(s), \;0<s<R/c$, and using the space-time divergence theorem we obtain, using also $e(0)=0$, that
$$
\begin{aligned}
e(s)\leq&\dfrac{-c}{\sqrt{1+c^{2}}}\; \int_{L} \left(e_K(x,t)+e_S(x,t)-F_1(x,t)\cdot \dfrac{x-x_{0}}{c|x-x_{0}|}\right)dS_{x,t}\\
&-\sum^{M}_{i=1} \int _{\Lambda (s)} \dfrac{\lambda^{i}(x)}{\gamma^i(x)} \left|
\int^{t}_{0}e^{-\frac{t-\tau}{\gamma^i(x)}}\nabla\cdot\vec{u}_{t}d\tau\right|^{2}dxdt\\
&-\sum^{N}_{i=1} \int _{\Lambda (s)} \frac{2 \mu^{i}(x)}{\tau^i(x)} \left| \int^{t}_{0}e^{-\frac{t-\tau}{\tau^i(x)}}
\epsilon_s(x,t)\,d\tau \right|^{2}dxdt,
\end{aligned}
$$
where
$$
\begin{aligned}
F_{1}(x,t):=&\left[\lambda^0(x) \nabla\cdot \vec{u} +\sum^{M}_{i=1}\int_0^t\lambda^i(x) e^{-\frac{t-s}{\gamma^i (x)}}\nabla\cdot\vec{u}_{s}(x,s)\,ds\right]\vec u_t \\
&+2\left[\mu^0(x) \epsilon(x,t)+\sum^{N}_{i=1} \int_0^t\mu^{i}(x) e^{-\frac{t-s}{\tau^i(x)}}\epsilon_s(x,s)\,ds\right]\vec{u}_t.
\end{aligned}
$$
Since it can be shown that
$e_K(x,t)+e_S(x,t)-F_1(x,t)\cdot \dfrac{x-x_{0}}{c|x-x_{0}|} \geq 0$, we get $0\le e(s)\le 0$,
implying $\vec{u}\equiv 0$ in $ \Lambda(R/c)$.  See \cite {MTY} for additional details in the above computations.
\end{proof}

Having given this background, we now present our main results in sections 3--5.

\section{Energy Lemmas}
We begin with a formal lemma that presents a work density decomposition in order to establish the
fundamental basis for our energy based approach. Throughout this paper we do not address smoothness considerations for
our functions $\sigma$, $e_K$, $e_S$, $e_D$, and $\vec u$. We assume all quantities exist and all differentiation and integrations
are achievable. Throughout sections 3-6 we develop results for $\mathbb{R}^3$; all the results can be
extended to $\mathbb{R}^d$, $d\ge 2$, using similar analysis. Often, we use the dots above letters to denote time derivatives.
\begin{lemma}\label{divergence}
Suppose $\vec u$ possesses the following work density decomposition:
\begin{equation}\label{work_density}
\nabla\cdot\left(\sigma(x,t) \vec u_t(x,t)\right)=\dot e_K(x,t)+\dot e_S(x,t)+\dot e_D(x,t)
\end{equation}
for any chosen $e_K$, $e_S$, and $e_D$.
For any open ball
$B_R(x_0)\subset \mathbb{R}^3$ and any given $c>0$, consider a space-time cone
$$
\forall s\in (0,R/c),\quad \Lambda(s):=\bigcup_{0<\tau<s} C_\tau
\quad\mbox{where}\quad C_\tau:=B_{R-c \tau}(x_0)\times\{t=\tau\},
$$
with $\partial \Lambda(s)=C_s\cup C_0\cup L$ with $L=\bigcup_{0<\tau<s} \partial C_\tau$.
Define the stored energy by
\begin{equation}\label{def_store}
e(s):=\int_{C_s} \left\{e_K(x,s)+e_S(x,s)\right\}\,dx,\quad \forall s\in [0,R/c).
\end{equation}
Then we have
\begin{equation}\label{ese0}
\begin{aligned}
e(s)-e(0)=&-\frac{c}{\sqrt{1+c^2}} \int_{L}
\left\{e_K(x,t)+e_S(x,t)-\frac{1}{c} (\sigma \vec u_t) \cdot \frac{x-x_0}{|x-x_0|}\right\}\,dS_{x,t}\\
&-\int_{\Lambda(s)} \dot e_D(x,t)\,dxdt.
\end{aligned}
\end{equation}
\end{lemma}
\begin{proof}
Integrating
$\nabla\cdot\left(\sigma \vec u_t\right)=\dot e_K(x,t)+\dot e_S(x,t)+\dot e_D(x,t)$ over $\Lambda(s)$, we get
$$
\int_{\Lambda(s)} \left(\dot e_K(x,t)+\dot e_S(x,t)
-\nabla\cdot\left(\sigma \vec u_t\right)\right)\,dxdt
=-\int_{\Lambda(s)}\dot e_D(x,t)\,dxdt.
$$
Applying space-time divergence theorem to the left hand side, we have
$$
LHS= \int_{\partial \Lambda(s)}
\left\{ \left(e_K(x,t)+e_S(x,t)\right)\nu_t
-(\sigma \vec u_t) \cdot \nu_x\right\}\,dS_{x,t},
$$
where $dS_{x,t}$ is the space-time boundary element,
and $(\nu_x,\nu_t)$ is the space-time outward normal to $\partial \Lambda(s)$ that is given by
$$
(\nu_x,\nu_t)=
\begin{cases}
  (0,1)  & \mbox{on  $C_s$},\\
  (0,-1) &\mbox{on  $C_0$},\\
  \displaystyle\frac{1}{\sqrt{1+c^2}}\left(\frac{x-x_0}{|x-x_0|},c \right)
  &\mbox{on $\displaystyle L$}.
\end{cases}
$$
Thus we have
$$
LHS=e(s)-e(0)+\frac{c}{\sqrt{1+c^2}} \int_{L}
\left\{e_K(x,t)+e_S(x,t)
-\frac{1}{c}  (\sigma \vec u_t)\cdot \frac{x-x_0}{|x-x_0|}\right\}\,dS_{x,t}.
$$
\end{proof}

\begin{remark}
$e_K$, $e_S$, and $e_D$ will be called the kinetic, strain, dissipated energy densities, and $\sigma$ will be the stress, which must be a symmetric
matrix. Then the physical meaning of \eqref{work_density} is the first law of thermodynamics (work-energy conservation);
the work done to a volume through its surface equals the sum of the change in kinetic and strain energy plus total dissipated energy.
However, Lemma \ref{divergence} does not assume more properties of energy densities such as nonnegativity; the lemma simply claims that
whenever the work density decomposition obeying the first law of thermodynamics holds, we have a useful identity \eqref{ese0} for the stored energy
difference.
\end{remark}

Now we consider a very general form of anisotropic viscoelastic linear constitutive equation for the stress-strain relation:
$$
\sigma(x,t)=\int_0^t \mathcal{C}(x,t,s) \epsilon_s(x,s)\,ds,\quad\mbox{where the strain is given by }
\epsilon=\frac 1 2 \left(\nabla\vec u+\nabla\vec u^T\right)
$$
and we assume three properties of the stiffness $4$-tensor $\mathcal{C}=(C_{ijkl})_{i,j,k,l=1}^3$: For any $x\in \bar \Omega$
where $\Omega$ is an open domain in $\mathbb{R}^3$ and $0\le s\le t\le T$,
\begin{itemize}
\item $\mathcal{C}(x,t,s)$ is symmetric: $C_{ijkl}=C_{klij}=C_{jikl}$, which already accounts for the symmetry of $\sigma$ and $\epsilon$.
\item $\mathcal{C}(x,t,s)$ is positive semi-definite, denoted by $\mathcal{C}(x,t,s)\succeq 0$ which means $M:\mathcal{C}(x,t,s)M\ge 0$ for all $3\times 3$ symmetric matrices $M$. Here the colon means componentwise inner product between two matrices. In our case, this semi-definite assumption is sufficient, because we assume the existence
    of the solution $\vec u$. However, standard methods for establishing existence and uniqueness of solutions would usually require that $\mathcal{C}(x,t,s)$ is positive definite.
\item $\mathcal{C}(x,t,t)$ is positive definite, denoted by $\mathcal{C}(x,t,t)\succ 0$ which means
$M:\mathcal{C}(x,t,t)M>0$ for all $3\times 3$ nonzero symmetric matrices $M$, or equivalently
all the eigenvalues $\lambda^{\mathcal{C}}(x,t)$ of $\mathcal{C}(x,t,t)$ are strictly positive. Here by the eigenvalue we mean
$$
\mathcal{C}(x,t,t)M=\lambda^{\mathcal{C}}(x,t) M\quad\mbox{for some $3\times 3$ nonzero symmetric  eigenmatrix $M$}.
$$
\end{itemize}
For actual computation of eigenvalues, we may convert $\mathcal{C}$ into a $6\times 6$ symmetric matrix and $M$ into $6\times 1$ vector using
Kelvin notation, which is designed to preserve the norms, multiplying $\sqrt{2}$ appropriately at off-diagonal places:
$$
\begin{aligned}
\mathcal{C}=(C_{ijkl})_{i,j,k,l=1}^3 &~\leftrightarrow~
\left(
  \begin{array}{crlrrr}
    C_{1111} & C_{1122} & C_{1133} & \sqrt{2}C_{1123} & \sqrt{2}C_{1131} & \sqrt{2}C_{1112} \\
     & C_{2222} & C_{2233} & \sqrt{2}C_{2223} & \sqrt{2}C_{2231} & \sqrt{2}C_{2212} \\
     & & C_{3333} & \sqrt{2}C_{3323} & \sqrt{2}C_{3331} & \sqrt{2}C_{3312} \\
     &  &  & 2C_{2323} & 2C_{2331} & 2C_{2312} \\
     & \text{sym.} &  &  &2C_{3131} & 2C_{3112} \\
     & & & & & 2C_{1212} \\
  \end{array}
\right),\\
M=(M_{ij})_{i,j=1}^3 &~\leftrightarrow~
(M_{11},M_{22},M_{33},\sqrt{2}M_{23},\sqrt{2}M_{31},\sqrt{2}M_{12})^T\quad\mbox{for symmetric matrix $M$}.
\end{aligned}
$$
Since the norms are preserved, the eigenvalues of this converted matrix are those of $\mathcal{C}$.
Note that more traditional Voigt notation conversion does not preserve the eigenvalues. See \cite{Helbig} for more details.

\begin{remark}
The positive (semi) definite assumption is sufficient for the Drucker stability criterion to be satisfied.
The Drucker stability criterion, see \cite{Bower}, pp. 67--68, is that (virtual) work done by the incremental changes of the stress and displacements to a volume through its surface is always nonnegative. The nonnegativity of this (virtual) work is not
a part of thermodynamic laws, but the Drucker criterion is utilized in many practical applications,
especially to establish the existence and uniqueness of stable solutions.
However, for large deformation cases especially in a nonlinear regime,
the strain may be opposite to the stress in the sense that some eigenvalues of $\mathcal{C}$ may be negative. Then the Drucker criterion is violated.
In our case, our main interest is wave propagation with small amplitude, so it
is reasonable to make the positive (semi) definite assumption.
\end{remark}
Thus, throughout the paper, we assume
\begin{equation}\label{C-cond}
\mbox{$\mathcal{C}(x,t,s)\succeq 0$ is symmetric and $\mathcal{C}(x,t,t)\succ 0$, for all
$x\in \bar \Omega$ and $0\le s\le t\le T$}.
\end{equation}
Then the inverse of $\mathcal{C}(x,t,t)$ exists and $\mathcal{S}(x,t):=\mathcal{C}^{-1}(x,t,t)$ becomes symmetric and positive definite
using the same criteria as for $\mathcal{C}$. As in the case of matrices ($2$-tensors), the eigenvalues of $\mathcal{S}(x,t)$ are $1/\lambda^{\mathcal{C}}(x,t)$ and
the minimum eigenvalue of $\mathcal{S}$ is $1/\lambda_{max}^{\mathcal{C}}(x,t)$, where $\lambda_{max}^{\mathcal{C}}(x,t)$
is the maximum eigenvalue of $\mathcal{C}(x,t,t)$. Also, the Rayleigh quotient provides
$$
\sigma:\mathcal{S}(x,t)\sigma\ge \frac{|\sigma|^2}{\lambda^{\mathcal{C}}_{max}(x,t)}\quad
\mbox{ for all $3\times 3$ symmetric matrices $\sigma$, where $|\sigma|^2=\sigma:\sigma$}.
$$
\begin{remark}
In applications, many materials are modeled by a convolution-type integral where $\mathcal{C}(x,t,s)=\mathcal{C}(x,t-s)$, which is called
a hereditary kernel. Our model is more general and it can also cover aging materials, see section \ref{sec:aging}.
\end{remark}
Now we consider the equation of motion, where we assume that the time independent mass density $\rho(x)\ge \rho_0$ for some $\rho_0>0$.

\begin{lemma}\label{lemma1}
Let $\vec u$ be a solution of
$$
\rho(x) \vec u_{tt} (x,t)=\nabla\cdot\sigma(x,t)
\quad\mbox{where}\quad\sigma(x,t)=\int_0^t \mathcal{C}(x,t,s) \epsilon_s(x,s)\,ds.
$$
Define the kinetic energy density, the strain energy density and the dissipated energy density by
\begin{equation}\label{def_energy}
\begin{aligned}
e_K(x,t)&=\frac{\rho(x)}{2}|\vec u_t(x,t)|^2,\qquad e_S(x,t)=\frac 1 2 \sigma(x,t): \mathcal{S}(x,t) \sigma(x,t),\\
\dot{e}_D(x,t)&=-\left[\sigma(x,t): \mathcal{S}(x,t) \tilde\sigma(x,t)+\frac 1 2 \sigma(x,t): \mathcal{S}_t(x,t) \sigma(x,t)\right],
\end{aligned}
\end{equation}
where $\displaystyle\tilde\sigma(x,t):=\int_0^t \mathcal{C}_t(x,t,s) \epsilon_s(x,s)\,ds$.
Then we have the work density decomposition given in \eqref{work_density}:
\begin{equation}\label{decomp_energy}
\nabla\cdot\left(\sigma \vec u_t\right)=\dot{e}_K+\sigma : \dot\epsilon=\dot{e}_K+\dot{e}_S+\dot{e}_D.
\end{equation}
\end{lemma}
\begin{proof}
Since $\sigma$ is symmetric, we have
$$
\begin{aligned}
\nabla\cdot\left(\sigma \vec u_t\right)&=
(\nabla\cdot \sigma) \cdot \vec u_t+\sigma : \nabla\vec u_t
=\rho \vec u_{tt}\cdot \vec u_t+\sigma :\dot\epsilon
=\left[\frac{\rho}{2}|\vec u_t|^2\right]_t+\sigma : \dot\epsilon
=\dot{e}_K+\sigma : \dot\epsilon.
\end{aligned}
$$
Since $\mathcal{S}$ is symmetric and $\dot\sigma=\mathcal{S}^{-1}\dot\epsilon+\tilde\sigma$,
we get $\sigma : \dot\epsilon=\dot{e}_S+\dot{e}_D$ as follows:
$$
\dot{e}_S=\frac{\sigma:\dot{\mathcal{S}}\sigma}{2}+\sigma:\mathcal{S}\dot{\sigma}
=\frac{\sigma:\dot{\mathcal{S}}\sigma}{2}+\sigma:\dot\epsilon+\sigma:\mathcal{S}\tilde\sigma
=\sigma:\dot{\epsilon}-\dot{e}_D.
$$
\end{proof}

Note that \eqref{def_energy} is not the only possible choice for $e_K$, $e_S$ and $\dot{e}_D$ to achieve the equation \eqref{decomp_energy}. The
 possibilities for alternate choices will be discussed in more detail in later sections.

\section{Main Finite Propagation Speed Result}\label{sec:fps}
We are now in a position to establish our main finite propagation speed results.

\begin{theorem}\label{elas_FPS_g}
Let $\vec u$ be a solution of
$$
\rho(x) \vec u_{tt} (x,t)=\nabla\cdot\sigma(x,t)\quad\mbox{ in } \Omega\times (0,T),
\quad\mbox{where}\quad\sigma(x,t)=\int_0^t \mathcal{C}(x,t,s) \epsilon_s(x,s)\,ds,
$$
and suppose
\begin{equation}\label{F-cond}
F(x,t):=\sigma(x,t): \mathcal{S}(x,t) \tilde\sigma(x,t)+\frac 1 2\sigma(x,t): \mathcal{S}_t(x,t) \sigma(x,t) \le 0.
\end{equation}
For any open ball $B_R(x_0)\subset \Omega$, if
$c=\sup_{(x,t)\in \bar B_{R} (x_0)\times [0,T]}\sqrt{{\lambda_{max}^{\mathcal{C}}(x,t)}/{\rho(x)}}<\infty$, then
$\vec u$ has finite propagation speed in $B_R(x_0)\times (0,T)$ with maximum propagation speed not exceeding $c$.
More precisely,
$$
\begin{aligned}
&\mbox{$\vec u(x,0)=\vec u_t(x,0)=0$ on $B_R(x_0)\subset \Omega$}\\
\Rightarrow\quad&
\mbox{$\vec u\equiv 0$ in the space-time cone $\Lambda(R/c)=\bigcup_{0<s<R/c} C_s$}.
\end{aligned}
$$
\end{theorem}
\begin{proof}
Start with defining $e_K$, $e_S$, $\dot{e}_D$, and $e$ as in \eqref{def_energy} and \eqref{def_store}. Then $e_K$, $e_S$, $e$, and $\dot{e}_D=-F$ are all nonnegative.
Also, by Lemma \ref{divergence} and \ref{lemma1}, \eqref{ese0} is valid:
$$
\begin{aligned}
e(s)-e(0)=&-\frac{c}{\sqrt{1+c^2}} \int_{L}
\left\{e_K(x,t)+e_S(x,t)-\frac{1}{c} (\sigma \vec u_t) \cdot \frac{x-x_0}{|x-x_0|}\right\}\,dS_{x,t}\\
&-\int_{\Lambda(s)} \dot e_D(x,t)\,dxdt,~\forall s\in [0,R/c).
\end{aligned}
$$
By the Cauchy-Schwarz inequality and the definition of $c$, the integrand of the first integral is nonnegative:
$$
\begin{aligned}
&e_K+e_S-\frac{1}{c} (\sigma \vec u_t) \cdot \frac{x-x_0}{|x-x_0|}\\
\ge&\frac{\rho}{2} |\vec u_t|^2 +\frac 1 2 \sigma: \mathcal{S}\sigma
-\frac{|\sigma||\vec u_t|}{c}
\ge\frac{1}{2}\left(\rho |\vec u_t|^2 +\frac{\left|\sigma\right|^2}{\lambda_{max}(x,t)}
-\frac{2|\sigma||\vec u_t|}{c}\right)\\
\ge&\frac{1}{2}\left(\rho |\vec u_t|^2 +\frac{\left|\sigma\right|^2}{\rho c^2}
-\frac{2|\sigma||\vec u_t|}{c}\right)
=\frac{1}{2}\left(\sqrt{\rho}|\vec u_t|-\frac{\left|\sigma\right|}{\sqrt{\rho}c}\right)^2\ge 0.
\end{aligned}
$$
Since we assume $\dot{e}_D=-F\ge 0$, we have $0\leq e(s)\le e(0)=0$, $\forall s\in [0,R/c)$ implying that
$\vec u(x,t)=0$ a.e. in the cone $\Lambda(R/c)$.
\end{proof}

Given $F$ as in \eqref{F-cond}, then the dissipated energy density introduced
in \eqref{def_energy} becomes $-F$. Then the inequality \eqref{F-cond} is nothing but the second law of thermodynamics.
Energy is always dissipating, which is a fundamental postulate in science. However, our choices of $e_S$ and $\dot{e}_D$ in \eqref{def_energy} are mathematical quantities that we make. In any given applications, our choice in  \eqref{def_energy} is not necessarily the actual physical strain energy density and dissipated energy  density. The point we make here is that  the constitutive equation itself does not uniquely determine
physical strain energy and dissipation;
one needs more information or an additional assumption to do so. We note, though, that there is no ambiguity in defining the kinetic energy; it is always $e_K=\frac{\rho}{2}|\vec u_t|^2$.
\begin{remark}\label{B_fps}
Accounting for this ambiguity, Theorem \ref{elas_FPS_g} may be further generalized;
whenever we are able to find a symmetric positive definite $\mathcal{B}(x,t)$ that satisfies
\begin{equation}\label{FB_dond}
F^\mathcal{B}(x,t)=\sigma: \big[\mathcal{B}\mathcal{S}^{-1}-\mathcal{I}\big] \dot\epsilon+\sigma: \mathcal{B} \tilde\sigma
+\frac 1 2 \sigma: \mathcal{B}_t \sigma \le 0,
\end{equation}
then $\vec u$ possesses finite propagation speed with the propagation speed not exceeding
$$
\sup_{(x,t)\in \bar B_{R} (x_0)\times [0,T]}\frac{1}{\sqrt{\rho(x)\lambda_{min}^{\mathcal{B}}(x,t)}}.
$$
In this case, when we define $e_S=\frac 1 2 \sigma:\mathcal{B}\sigma$ and $\dot{e}_D=-F^{\mathcal{B}}$, then $\sigma : \dot\epsilon=\dot{e}_S+\dot{e}_D$ implies the work
density decomposition \eqref{decomp_energy}, and thus the finite propagation speed property.
In Theorem \ref{elas_FPS_g}, we pick $\mathcal{B}=\mathcal{S}$ to make the first term of $F^\mathcal{B}(x,t)$ vanish.
\end{remark}

\begin{remark}\label{cond_suff_F}
Here we give a sufficient condition for  \eqref{F-cond}:

If
$\mathcal{C}(x,t,s)= e^{-t\mathcal{A}(x)}\mathcal{C}_0(x,s)$ for some $\mathcal{C}_0$ and $\mathcal{A}$ satisfying
\begin{itemize}
\item $\mathcal{C}_0$ is symmetric positive definite, $\dot{\mathcal{C}}_0$ is symmetric positive semi-definite and $\mathcal{A} \dot{\mathcal{C}}_0=\dot{\mathcal{C}}_0\mathcal{A}$,
\item $\mathcal{A}$ is symmetric positive semi-definite and $\mathcal{C}_0 \mathcal{A}=\mathcal{A}\mathcal{C}_0$,
\end{itemize}
then \eqref{F-cond} is satisfied, because
$F=\sigma: \mathcal{S}\mathcal{A} \sigma-\frac 1 2 \sigma: \mathcal{S}\dot{\mathcal{C}} \mathcal{S} \sigma$ and $\dot{\mathcal{C}}=
-\mathcal{A}\mathcal{C}+e^{-t\mathcal{A}} \dot{\mathcal{C}}_0$ imply
$\displaystyle F=-\frac 1 2 \sigma:[\mathcal{A}\mathcal{S}+\mathcal{S}e^{t\mathcal{A}}\dot{C}_0 \mathcal{S}]\sigma$.
\end{remark}

\section{Superposition Principle}
Many viscoelastic models including spring-dashpot models consist of a chain of small elementary viscoelastic units.
In this case, the stiffness tensor is in a linear superposition form:
\begin{equation}\label{C_superposition}
\mathcal{C}(x,t,s)=\sum_{i=1}^n \mathcal{C}^i(x,t,s),
\end{equation}
\begin{equation}\label{sigma_superposition}
\sigma(x,t)=\int_0^t \sum_{i=1}^n\mathcal{C}^i(x,t,s) \epsilon_s(x,s)\,ds,
\end{equation}
where each $\mathcal{C}^i$ satisfies \eqref{C-cond}. In this situation, rather than considering $\mathcal{C}$ as a whole, it may be more
convenient to regard \eqref{C_superposition} and \eqref{sigma_superposition} as a summation of simple units: Define
$\mathcal{S}^i(x,t)=[\mathcal{C}^i(x,t,t)]^{-1}$,
$$
\begin{gathered}
\sigma^i(x,t)=\int_0^t \mathcal{C}^i(x,t,s) \epsilon_s(x,s)\,ds,\quad \tilde\sigma^i(x,t)=\int_0^t \mathcal{C}^i_t(x,t,s) \epsilon_s(x,s)\,ds,\\
e_S^i(x,t)=\frac 1 2 \sigma^i(x,t): \mathcal{S}^i(x,t) \sigma^i(x,t),\\
\dot{e}^i_D(x,t)=-\left[\sigma^i(x,t): \mathcal{S}^i(x,t) \tilde\sigma^i(x,t)+\frac 1 2 \sigma^i(x,t): \mathcal{S}^i_t(x,t) \sigma^i(x,t)\right].
\end{gathered}
$$
Then we have a superposed version of Lemma \ref{lemma1}.
\begin{lemma}\label{lemma1_superposition}
Let $\vec u$ be a solution of
$$
\rho(x) \vec u_{tt} (x,t)=\nabla\cdot\sigma(x,t)
\quad\mbox{where}\quad\sigma(x,t)=\sum_{i=1}^n\int_0^t \mathcal{C}^i(x,t,s) \epsilon_s(x,s)\,ds.
$$
Define the kinetic energy density, the strain energy density and the dissipated energy density by
\begin{equation}\label{def_energy_superpose}
e_K(x,t)=\frac{\rho(x)}{2}|\vec u_t(x,t)|^2,\quad
e_S(x,t)=\sum_{i=1}^n e_S^i(x,t),\quad
\dot{e}_D(x,t)=\sum_{i=1}^n \dot{e}_D^i(x,t).
\end{equation}
Then we have the work density decomposition,
$\nabla\cdot\left(\sigma \vec u_t\right)=\dot{e}_K+\sigma : \dot\epsilon=\dot{e}_K+\dot{e}_S+\dot{e}_D$.
\end{lemma}
\begin{proof}
Since $\sigma=\sum_{i=1}^n \sigma^i$ and $\epsilon$ is independent of $i$, just as in the proof of Lemma \ref{lemma1}, we get
$\displaystyle
\nabla\cdot\left(\sigma \vec u_t\right)=\dot{e}_K+\sigma : \dot\epsilon=\dot{e}_K+\sum_{i=1}^n \sigma^i : \dot\epsilon
=\dot{e}_K+\sum_{i=1}^n \left(\dot{e}_S^i+\dot{e}_D^i\right)$.
\end{proof}

\begin{theorem}\label{elas_FPS_s}
Let $\vec u$ be a solution of
$$
\rho(x) \vec u_{tt} (x,t)=\nabla\cdot\sigma(x,t)\quad\mbox{ in } \Omega\times (0,T),
\quad\mbox{where}\quad\sigma(x,t)=\sum_{i=1}^n\int_0^t \mathcal{C}^i(x,t,s) \epsilon_s(x,s)\,ds
$$
and suppose
\begin{equation}\label{F_cond_superposition}
\begin{aligned}
F_{sum}(x,t):=&-\dot{e}_D(x,t)\\
=&\sum_{i=1}^n\left(\sigma^i(x,t): \mathcal{S}^i(x,t) \tilde\sigma^i(x,t)+\frac 1 2 \sigma^i(x,t): \mathcal{S}^i_t(x,t) \sigma^i(x,t)\right)
\le 0.
\end{aligned}
\end{equation}
$$
$$
For any open ball $B_R(x_0)\subset \Omega$, if
$\displaystyle c=\sup_{(x,t)\in \bar B_{R} (x_0)\times [0,T]}\sqrt{{\sum_{i=1}^n\lambda_{max}^{\mathcal{C}^i}(x,t)}/{\rho(x)}}<\infty$, then
$\vec u$ has finite propagation speed in $B_R(x_0)\times (0,T)$ with maximum propagation speed not exceeding $c$.
\end{theorem}
\begin{proof}
Let $\lambda(x,t)=\sum_{i=1}^n\lambda^{\mathcal{C}^i}_{max}(x,t)$. Then
\begin{align*}
&e_K(x,t)+e_S(x,t)-\frac{1}{c} (\sigma \vec u_t) \cdot \frac{x-x_0}{|x-x_0|}
=\frac{\rho}{2}|\vec u_t|^2+\sum_{i=1}^n \left(e_S^i(x,t)-\frac{1}{c} (\sigma^i \vec u_t) \cdot \frac{x-x_0}{|x-x_0|}\right)\\
=&\sum_{i=1}^n \left(\frac{\rho  \lambda_{max}^{\mathcal{C}^i}}{2\lambda}|\vec u_t|^2+\frac 1 2 \sigma^i: \mathcal{S}^i \sigma^i-\frac{1}{c} (\sigma^i \vec u_t) \cdot \frac{x-x_0}{|x-x_0|}\right).
\end{align*}
Recalling the proof of Theorem \ref{elas_FPS_g}, it suffices to show the above is nonnegative. Here we show each component is nonnegative:
the above parenthesis is larger than or equal to
$$
\begin{aligned}
&\frac{\rho  \lambda_{max}^{\mathcal{C}^i}}{2\lambda}|\vec u_t|^2+\frac 1 2 \sigma^i: \mathcal{S}^i \sigma^i-\frac{|\sigma^i||\vec u_t|}{c}
\ge\frac{1}{2}\left(\frac{\rho  \lambda_{max}^{\mathcal{C}^i}}{\lambda} |\vec u_t|^2 +\frac{|\sigma^i|^2}{\lambda_{max}^{\mathcal{C}^i}}
-\frac{2|\sigma^i||\vec u_t|}{c}\right)\\
=&\frac{\lambda_{max}^{\mathcal{C}^i}}{2}\left(\frac{\rho}{\lambda} |\vec u_t|^2 +\frac{|\sigma^i|^2}{(\lambda_{max}^{\mathcal{C}^i})^2}
-\frac{2|\sigma^i||\vec u_t|}{c\lambda_{max}^{\mathcal{C}^i}}\right)\ge\frac{\lambda_{max}^{\mathcal{C}^i}}{2}\left(\frac{|\vec u_t|^2}{c^2}  +\frac{|\sigma^i|^2}{(\lambda_{max}^{\mathcal{C}^i})^2}
-\frac{2|\sigma^i||\vec u_t|}{c\lambda_{max}^{\mathcal{C}^i}}\right)\\
=&\frac{\lambda_{max}^{\mathcal{C}^i}}{2}\left(\frac{|\vec u_t|}{c}-\frac{|\sigma^i|}{\lambda_{max}^{\mathcal{C}^i}}
\right)^2\ge 0.
\end{aligned}
$$
\end{proof}

Now we consider a continuous superposition,
\begin{equation}\label{C_superposition_int}
\mathcal{C}(x,t,s)=\int_0^\infty \mathcal{C}^\tau(x,t,s)d\beta(\tau),
\end{equation}
where $\beta$ is a finite Borel measure and $\mathcal{C}^\tau$ satisfies \eqref{C-cond} for each $\tau$.
Then our smoothness assumptions imply
\begin{equation}\label{sigma_superposition_int}
\sigma(x,t)=\int_0^t \int_0^\infty\mathcal{C}^\tau(x,t,s)d\beta(\tau) \epsilon_s(x,s)\,ds
=\int_0^\infty \int_0^t\mathcal{C}^\tau(x,t,s) \epsilon_s(x,s)\,dsd\beta(\tau),
\end{equation}
and, as before, we define $\mathcal{S}^\tau(x,t)=[\mathcal{C}^\tau(x,t,t)]^{-1}$,
$$
\begin{gathered}
\sigma^\tau(x,t)=\int_0^t \mathcal{C}^\tau(x,t,s) \epsilon_s(x,s)\,ds,\quad \tilde\sigma^\tau(x,t)=\int_0^t \mathcal{C}^\tau_t(x,t,s) \epsilon_s(x,s)\,ds,\\
e_S^\tau(x,t)=\frac 1 2 \sigma^\tau(x,t): \mathcal{S}^\tau(x,t) \sigma^\tau(x,t),\\
\dot{e}_D^\tau(x,t)=-\left[\sigma^\tau(x,t): \mathcal{S}^\tau(x,t) \tilde\sigma^\tau(x,t)+\frac 1 2 \sigma^\tau(x,t): \mathcal{S}_t^\tau(x,t) \sigma^\tau(x,t)\right].
\end{gathered}
$$
Then we have an integral version of Lemma \ref{lemma1}.
\begin{lemma}\label{lemma1_superposition_int}
Let $\vec u$ be a solution of
$$
\rho(x) \vec u_{tt} (x,t)=\nabla\cdot\sigma(x,t)
\quad\mbox{where}\quad\sigma(x,t)=\int_0^\infty \int_0^t \mathcal{C}^\tau(x,t,s) \epsilon_s(x,s)\,ds d\beta(\tau).
$$
Define the kinetic energy density, the strain energy density and the dissipated energy density by
\begin{equation}\label{def_energy_superpose_int}
e_K(x,t)=\frac{\rho(x)}{2}|\vec u_t(x,t)|^2,\quad
e_S(x,t)=\int_0^\infty e_S^\tau(x,t)d\beta(\tau),\quad
\dot{e}_D(x,t)=\int_0^\infty \dot{e}_D^\tau(x,t)d\beta(\tau).
\end{equation}
Then we have the work density decomposition,
$\nabla\cdot\left(\sigma \vec u_t\right)=\dot{e}_K+\sigma : \dot\epsilon=\dot{e}_K+\dot{e}_S+\dot{e}_D$.
\end{lemma}
\begin{proof}
As in the proof of Lemma \ref{lemma1}, we have
$\nabla\cdot\left(\sigma \vec u_t\right)=\dot{e}_K+\sigma : \dot\epsilon$, and also for each fixed $\tau$,
$$
\sigma^\tau(x,t):\dot\epsilon(x,t)=\dot e_S^\tau(x,t)+\dot e_D^\tau(x,t).
$$
Since $\epsilon$ is independent of $\tau$, integrating over $\tau$, we get
$\sigma:\dot\epsilon = \dot e_S+\dot e_D$.
\end{proof}

\begin{theorem}\label{elas_FPS_s_int}
Let $\vec u$ be a solution of
$$
\rho(x) \vec u_{tt} (x,t)=\nabla\cdot\sigma(x,t)\quad\mbox{ in } \Omega\times (0,T),
$$
where $\sigma(x,t)=\int_0^t \int_0^\infty\mathcal{C}^\tau(x,t,s)d\beta(\tau) \epsilon_s(x,s)\,ds$
and suppose
\begin{equation}\label{F_cond_superposition_int}
\begin{aligned}
F_{int}(x,t):=&-\dot{e}_D(x,t)\\
=&\int_0^\infty\left(\sigma^\tau(x,t): \mathcal{S}^\tau(x,t) \tilde\sigma^\tau(x,t)+\frac 1 2 \sigma^\tau(x,t): \mathcal{S}^\tau_t(x,t) \sigma^\tau(x,t)\right)d\beta(\tau)\le 0.
\end{aligned}
\end{equation}
Let $\lambda(x,t)=\int_0^\infty \lambda_{max}^{\mathcal{C}^\tau}(x,t)d\beta(\tau)$ and suppose
$c=\sup_{(x,t)\in \bar B_{R} (x_0)\times [0,T]}\sqrt{{\lambda(x,t)}/{\rho(x)}}<\infty$. Then
for any open ball $B_R(x_0)\subset \Omega$,
$\vec u$ has finite propagation speed in $B_R\times (0,T)$ with maximum propagation speed not exceeding $c$.
\end{theorem}
\begin{proof}
As before,
\begin{align*}
&e_K(x,t)+e_S(x,t)-\frac{1}{c} (\sigma \vec u_t) \cdot \frac{x-x_0}{|x-x_0|}\\
=&\frac{\rho}{2}|\vec u_t|^2+\int_0^\infty \left(e_S^\tau(x,t)-\frac{1}{c} (\sigma^\tau(x,t) \vec u_t) \cdot \frac{x-x_0}{|x-x_0|}\right)d\beta(\tau)\\
=&\int_0^\infty \left(\frac{\rho \lambda_{max}^{\mathcal{C}^\tau}(x,t)}{2\lambda(x,t)}|\vec u_t|^2+e_S^\tau(x,t)-\frac{1}{c} (\sigma^\tau(x,t) \vec u_t) \cdot \frac{x-x_0}{|x-x_0|}\right)d\beta(\tau).
\end{align*}
Recalling the proof of Theorem \ref{elas_FPS_g}, it suffices to show the above is nonnegative. Here we show for each $\tau$ the integrand is nonnegative:
the above parenthesis is larger than or equal to
$$
\begin{aligned}
&\frac{\rho  \lambda_{max}^{\mathcal{C}^\tau}(x,t)}{2\lambda(x,t)}|\vec u_t|^2
+\frac 1 2 \sigma^\tau(x,t): \mathcal{S}^\tau(x,t) \sigma^\tau(x,t)-\frac{|\sigma^\tau(x,t)||\vec u_t|}{c}\\
\ge &\frac{1}{2}\left(\frac{\rho  \lambda_{max}^{\mathcal{C}^\tau}(x,t)}{\lambda(x,t)}|\vec u_t|^2 +\frac{|\sigma^\tau(x,t)|^2}{\lambda_{max}^{\mathcal{C}^\tau}(x,t)}
-\frac{2|\sigma^\tau(x,t)||\vec u_t|}{c}\right)\\
=&\frac{\lambda_{max}^{\mathcal{C}^\tau}(x,t)}{2}\left(\frac{\rho}{\lambda} |\vec u_t|^2 +\frac{|\sigma^\tau(x,t)|^2}{(\lambda_{max}^{\mathcal{C}^\tau})^{2}}
-\frac{2|\sigma^\tau(x,t)||\vec u_t|}{c\lambda_{max}^{\mathcal{C}^\tau}}\right)\\
\ge&\frac{\lambda_{max}^{\mathcal{C}^\tau}}{2}\left(\frac{|\vec u_t|^2}{c^2}+\frac{|\sigma(x,t;\tau)|^2}{(\lambda_{max}^{\mathcal{C}^\tau})^{2}}
-\frac{2|\sigma^\tau(x,t)||\vec u_t|}{c\lambda_{max}^{\mathcal{C}^\tau}}\right)
=\frac{\lambda_{max}^{\mathcal{C}^\tau}}{2}\left(\frac{|\vec u_t|}{c}-\frac{|\sigma^\tau(x,t)|}{\lambda_{max}^{\mathcal{C}^\tau}}\right)^2\ge 0.
\end{aligned}
$$
\end{proof}

\begin{remark}
If each $\mathcal{C}^i$
satisfies \eqref{F-cond}, then $\mathcal{C}=\sum_{i=1}^n \mathcal{C}^i$ automatically satisfies \eqref{F_cond_superposition}.
Therefore, this superposition principle is extremely useful when
solutions, $\vec{u}^i$, of $\rho \vec u_{tt}^i =\nabla\cdot\sigma^i$, are known to possess the finite propagation speed property for
$i=1,\ldots,n$ (For the continuous superposition case, the same argument holds if \eqref{F-cond} is satisfied for each $\tau$).
However, the estimated maximum propagation speed through the superposition principle may be overestimated. To obtain a shaper estimate, we have to consider the kernel $\mathcal{C}$ as a whole. In this case, we have a different strain
energy density, $e_S=\frac 1 2 \sigma: \mathcal{S} \sigma$,
where the needed properties of $\mathcal{S}=(\mathcal{C}^1+\cdots+\mathcal{C}^n)^{-1}$ in most cases cannot be straightforwardly computed
in terms of those same properties of $\mathcal{C}^i$.
The difficulty lies mainly in the fact that $(\mathcal{C}^1+\cdots+\mathcal{C}^n)^{-1}\neq (\mathcal{C}^1)^{-1}+\cdots+(\mathcal{C}^n)^{-1}$.
\end{remark}

Combining the superposition principles with Remark \ref{cond_suff_F}, we immediately see the chain of exponential (in $t$) type stiffness tensors
(including the generalized Linear Solid Models) exhibits finite propagation speed.

\begin{corollary}\label{exponential_kernel}
Suppose the stiffness kernel is given by the following exponential form in time $t$,
\begin{equation}\label{C_sum_exp}
\mathcal{C}(x,t,s)=\sum_{i=1}^n \mathcal{C}^i(x,t,s)+\int_0^\infty e^{-t \mathcal{A}^\tau(x)} \mathcal{C}_0^\tau(x,s)d\beta(\tau),
\end{equation}
where each $\mathcal{C}^i$
satisfies \eqref{F-cond} and, for each $\tau$, $\mathcal{A}^\tau(x)$ and $\mathcal{C}_0^\tau(x,s)$ satisfy the conditions given in
Remark \ref{cond_suff_F}. Then
$\vec{u}$ has finite propagation speed with maximum propagation speed not exceeding
$$
c=\sup_{(x,t)\in \bar B_{R} (x_0)\times [0,T]}\left[\frac{\sum_{i=1}^n\lambda_{max}^{\mathcal{C}^i}(x,t)+\int_0^\infty \lambda_{max}^\tau(x,t)d\beta(\tau)}{\rho(x)}\right]^{\frac 1 2}.
$$
Here $\lambda_{max}^\tau(x,t)$ is the maximum eigenvalue of $e^{-t \mathcal{A}^\tau(x)} \mathcal{C}_0^\tau(x,t)$.
\end{corollary}
Here the measure $\beta$ can be a finite sum of Dirac deltas.

\section{Examples}
We start with isotropic models that have the following form of stiffness tensor,
$$
\mathcal{C}=\lambda I\otimes I + 2\mu\mathcal{I},
$$
where $I$ and $\mathcal{I}$ denote the identity matrix and the identity $4$-tensor, respectively.

\subsection{Purely elastic isotropic case (Time independent stiffness tensor)}
Consider
$$
\mathcal{C}(x,t,s)=\mathcal{C}(x)=\lambda(x) I\otimes I + 2\mu(x)\mathcal{I},
$$
where $\lambda(x),\mu(x)>0$. This case was already established in \cite{MY}. We recast this example using the tool we developed in this paper.
Since $\mathcal{C}_t$, $\mathcal{S}_t$, $\tilde \sigma$ vanish, obviously we get $F=0$.
So we have finite propagation speed with
$\displaystyle
c=\sup_{x\in \bar B_{R} (x_0)} \sqrt{\frac{\lambda(x)+2\mu(x)}{\rho(x)}}$.

\subsection{Stiffness tensor of exponential convolution type}
Consider
$$
\mathcal{C}(x,t,s)=\mathcal{C}(x,t-s)=\lambda(x)e^{-\frac{t-s}{\gamma(x)}} I\otimes I + 2\mu(x) e^{-\frac{t-s}{\tau(x)}}\mathcal{I},
$$
where $\lambda(x),\mu(x),\gamma(x),\tau(x)>0$. Then we have
\begin{align*}
\mathcal{C}_t(x,t,s)&=-\frac{\lambda(x)}{\gamma(x)}e^{-\frac{t-s}{\gamma(x)}} I\otimes I -\frac{2\mu(x)}{\tau(x)} e^{-\frac{t-s}{\tau(x)}}\mathcal{I},\\
\mathcal{C}(x,t,t)&=\lambda(x) I\otimes I + 2\mu(x) \mathcal{I},\\
\mathcal{S}(x,t)&=\mathcal{C}^{-1}(x,t,t)=-\frac{\lambda(x)}{2\mu(x)(3\lambda(x)+2\mu(x))} I\otimes I + \frac{1}{2\mu(x)} \mathcal{I},\\
\mathcal{S}_t(x,t)&=0.
\end{align*}
If we define the strain energy density as suggested in \eqref{def_energy}, then we find
$$
\begin{aligned}
&e_S(x,t)=\frac 1 2 \sigma: \mathcal{S} \sigma\\
=&
\frac{1}{\mu}\left(\left|\sigma_{\mu,\tau}\right|^2-\frac{\lambda}{3\lambda+2\mu}\left|\mbox{tr}(\sigma_{\mu,\tau})\right|^2\right)
+\frac{\frac 3 2}{3\lambda+2\mu}\left|\mbox{tr}(\sigma_{\lambda,\gamma})\right|^2
+\frac{2}{3\lambda+2\mu}\mbox{tr}(\sigma_{\lambda,\gamma})\mbox{tr}(\sigma_{\mu,\tau}),
\end{aligned}
$$
where
$\displaystyle\sigma_{\lambda,\gamma}(x,t):=\int_0^t \lambda(x)e^{-\frac{t-s}{\gamma(x)}}\epsilon_s(x,s)\,ds$,
$\displaystyle\sigma_{\mu,\tau}(x,t):=\int_0^t \mu(x)e^{-\frac{t-s}{\tau(x)}}\epsilon_s(x,s)\,ds$ and so
$\sigma=\mbox{tr}(\sigma_{\lambda,\gamma})I+2\sigma_{\mu,\tau}$. Unfortunately, this choice of the strain energy
density does not satisfy the energy dissipation criterion \eqref{F-cond} with
$$
\begin{aligned}
-F(x,t)=-\sigma: \mathcal{S} \tilde{\sigma}=&
\frac{2}{\tau\mu}\left|\sigma_{\mu,\tau}\right|^2+\frac{3}{\gamma(3\lambda+2\mu)}\left|\mbox{tr}(\sigma_{\lambda,\gamma})\right|^2\\
&+\frac{2(\gamma+\tau)}{\gamma\tau(3\lambda+2\mu)}\mbox{tr}(\sigma_{\lambda,\gamma})\mbox{tr}(\sigma_{\mu,\tau})
-\frac{2\lambda}{\gamma\tau(3\lambda+2\mu)}
\left|\mbox{tr}(\sigma_{\mu,\tau})\right|^2,
\end{aligned}
$$
which can be negative for $\gamma\gg 1$. So we define our strain energy density differently, but consistent
with our new method:
$$
\begin{aligned}
e_S(x,t)=&
\frac{1}{\mu}\left|\sigma_{\mu,\tau}\right|^2
+\frac{1}{2\lambda}\left|\mbox{tr}(\sigma_{\lambda,\gamma})\right|^2,
\end{aligned}
$$
which is identical to that in Section \ref{sec:Background}. Then we get
$$
\sigma:\dot \epsilon-\dot{e}_S=\left(\mbox{tr}(\sigma_{\lambda,\gamma})I+2\sigma_{\mu,\tau}\right):\dot{\epsilon}
-\frac{2}{\mu}\sigma_{\mu,\tau}:\dot\sigma_{\mu,\tau}-\frac{1}{\lambda}\mbox{tr}(\sigma_{\lambda,\gamma})\mbox{tr}(\dot\sigma_{\lambda,\gamma}).
$$
Observing that $\dot\sigma_{\lambda,\gamma}=\lambda\dot\epsilon-\frac{1}{\gamma}\sigma_{\lambda,\gamma}$ and
$\dot\sigma_{\mu,\tau}=\mu\dot\epsilon-\frac{1}{\tau}\sigma_{\mu,\tau}$, we have
$$
\sigma:\dot \epsilon-\dot{e}_S=\frac{2}{\mu\tau}|\sigma_{\mu,\tau}|^2+\frac{1}{\lambda\gamma}|\mbox{tr}(\sigma_{\lambda,\gamma})|^2.
$$
Therefore, if we define $\dot{e}_D:=\frac{2}{\mu\tau}|\sigma_{\mu,\tau}|^2+\frac{1}{\lambda\gamma}|\mbox{tr}(\sigma_{\lambda,\gamma})|^2$,
it possesses the work density decomposition \eqref{decomp_energy} given in Lemma \ref{lemma1}. Furthermore,
$$
\begin{aligned}
&F=-\dot{e}_D\\
=&\frac{-1}{\gamma(x)\lambda(x)}\left|\int_0^t \lambda(x)e^{-\frac{t-s}{\gamma(x)}}\mbox{tr}(\epsilon_s(x,s))\,ds\right|^2
+\frac{-2}{\tau(x)\mu(x)}\left|\int_0^t\mu(x)e^{-\frac{t-s}{\tau(x)}}\epsilon_s(x,s)\,ds\right|^2\le 0.
\end{aligned}
$$
Then we can follow the proof of Theorem \ref{elas_FPS_g} to achieve finite propagation speed with
$\displaystyle c=\sup_{x\in \bar B_{R} (x_0)} \sqrt{\frac{\lambda(x)+2\mu(x)}{\rho(x)}}$,
where we additionally need in the proof the following inequality:
$$
\begin{aligned}
&e_K(x,t)+e_S(x,t)-\frac{1}{c} (\sigma \vec u_t) \cdot \frac{x-x_0}{|x-x_0|}\\
\ge&\frac{\rho}{2} |\vec u_t|^2 +\frac{\lambda\left|\nabla\cdot\vec{u}\right|^2}{2}
+\mu|\epsilon|^2-\frac{1}{c} (\lambda|\nabla\cdot\vec u||\vec u_t|+2\mu|\epsilon||\vec u_t|)\\
\ge& \frac{\lambda}{2}\left(\frac{\rho|\vec u_t|^2}{\lambda+2\mu}+|\nabla\cdot\vec u|^2
-\frac{2 |\nabla\cdot \vec u||\vec u_t|}{c}\right)
+\mu\left(\frac{\rho|\vec u_t|^2}{\lambda+2\mu}+|\epsilon|^2-\frac{2|\epsilon||\vec u_t|}{c}\right)\\
\ge& \frac{\lambda}{2}\left(\frac{|\vec u_t|^2}{c^2}+|\nabla\cdot\vec u|^2
-\frac{2 |\nabla\cdot \vec u||\vec u_t|}{c}\right)
+\mu\left(\frac{|\vec u_t|^2}{c^2}+|\epsilon|^2-\frac{2|\epsilon||\vec u_t|}{c}\right)\\
=& \frac{\lambda(x,t)}{2}\left(\frac{|\vec u_t|}{c}-|\nabla\cdot\vec u|\right)^2
+\mu(x,t)\left(\frac{|\vec u_t|}{c}-|\epsilon|\right)^2\ge 0.
\end{aligned}
$$
This example illustrates that the choice of strain energy density to achieve a finite
propagation speed result may be model dependent.

\subsection{Generalized Linear Solid model (Wiechert Model)}
Consider
$$
\begin{aligned}
&\mathcal{C}(x,t,s)=\mathcal{C}(x,t-s)\\
=&\left(\lambda^0(x)+\sum_{i=1}^n\lambda^i(x)e^{-\frac{t-s}{\gamma^i(x)}}\right) I\otimes I + 2\left(\mu^0(x)+\sum_{i=1}^n \mu^i(x) e^{-\frac{t-s}{\tau^i(x)}}\right)\mathcal{I},
\end{aligned}
$$
where
$\gamma^i(x),\tau^i(x),\lambda^i(x),\mu^i(x)>0$ for all $i=0,1,\ldots,n$. Then
$\mathcal{C}$ can be naturally decomposed into a sum of tensors where each individual tensor yields
finite propagation speed. By the superposition principle, then $\mathcal{C}$ also yields finite propagation speed with
$\displaystyle c=\sup_{x\in \bar B_{R} (x_0)} \sqrt{\frac{\sum_{i=0}^n(\lambda^i(x)+2\mu^i(x))}{\rho(x)}}$.
Comparing to the example in Section \ref{sec:Background}, here we have chosen $n=M=N$, but it is not necessary to do so.
A straightforward restructuring of the elasticity tensor leads to a similar result.

\subsection{Aging isotropic material}\label{sec:aging}
If a medium is being aged during wave propagation, the constitutive law is not of convolution type.
Since the time scale for aging is slow relative to the wave dynamics, we may model it as
$$
\mathcal{C}(x,t,s)=\lambda(x,t) I\otimes I + 2\mu(x,t) \mathcal{I},
$$
where $\lambda(x,t),\mu(x,t)>0$ and $\lambda_t(x,t),\mu_t(x,t)\le 0$.
Define the strain energy density as
$$
e_S(x,t)=\frac{\lambda(x,t)}{2}|\mbox{tr}(\epsilon)|^2 +\mu(x,t) |\epsilon|^2,
$$
motivated again by the purely elastic case. Then we get
$$
\sigma:\dot \epsilon-\dot{e}_S=-\frac 1 2 \big(\lambda_t |\mbox{tr}(\epsilon)|^2+2\mu_t|\epsilon|^2\big).
$$
Therefore, if we define $\dot{e}_D:=-\frac 1 2 \big(\lambda_t |\mbox{tr}(\epsilon)|^2+2\mu_t|\epsilon|^2\big)\ge 0$,
it possesses the work density decomposition \eqref{decomp_energy} given in Lemma \ref{lemma1}, and $F=-\dot{e}_D\le 0$.
Then we can follow the proof of Theorem \ref{elas_FPS_g} to achieve finite propagation speed with
$$
c=\sup_{(x,t)\in \bar B_{R} (x_0)\times [0,T]} \sqrt{\frac{\lambda(x,t)+2\mu(x,t)}{\rho(x)}}.
$$

\subsection{Sum of aging and exponential convolution type models}
Consider
$$
\mathcal{C}(x,t,s)=\left(\lambda^0(x,t)+\sum_{i=1}^n\lambda^i(x)e^{-\frac{t-s}{\gamma^i(x)}}\right) I\otimes I + 2\left(\mu^0(x,t)+\sum_{i=1}^n \mu^i(x) e^{-\frac{t-s}{\tau^i(x)}}\right)\mathcal{I},
$$
where $\lambda^0(x,t),\mu^0(x,t)>0$, $\lambda_t^0(x,t),\mu_t^0(x,t)\le 0$, and $\gamma^i(x),\tau^i(x),\lambda^i(x),\mu^i(x)>0$ for all $i=0,1,\ldots,n$.
Then
$\mathcal{C}$ can be naturally decomposed into a sum of tensors where each individual tensor yields
finite propagation speed. By the superposition principle, then $\mathcal{C}$ also yields finite propagation speed with
$$
\displaystyle c=\sup_{(x,t)\in \bar B_{R} (x_0)\times [0,T]} \sqrt{\frac{\lambda^0(x,t)+2\mu^0(x,t)+\sum_{i=1}^n
(\lambda^i(x)+2\mu^i(x))}{\rho(x)}}.
$$
Comparing to the example in Section \ref{sec:Background}, here $\lambda^0$ and $\mu^0$ are allowed to be time-dependent.

\subsection{Fractional Zener model}
Consider
\begin{equation}\label{frac_zener}
\mathcal{C}(x,t,s)=\mathcal{C}^1(x)+E_\alpha\left[-\left(\frac{t-s}{a}\right)^\alpha\right] \mathcal{M}(x), \quad 0<\alpha<1,\quad
a>0,
\end{equation}
where $\mathcal{C}^1$ and $\mathcal{M}$ are symmetric and positive definite, and
$E_\alpha(z)=\sum_{n=0}^\infty\frac{z^n}{\Gamma(\alpha n+1)}$ is the Mittag-Leffler function.
It is known that $E_\alpha[-{z^\alpha}/{a^\alpha}]$ is completely monotone on $(0,\infty)$, thus it possesses the following Laplace transform  representation
by Berstein's theorem:
$$
\begin{aligned}
E_\alpha\left[-\left(\frac{t-s}{a}\right)^\alpha\right]=&\int_0^\infty e^{-(t-s)\tau} d\beta(\tau)\\
=&\frac{\sin \alpha \pi}{\pi} \int_0^\infty e^{-(t-s)\tau}\frac{(a\tau)^\alpha}{\tau[(a\tau)^{2\alpha}+2(a\tau)^\alpha \cos \alpha\pi+1]}d\tau,
\end{aligned}
$$
where the measure $\beta$ can be explicitly computed in this case. Therefore
$$
\begin{aligned}
\mathcal{C}(x,t,s)&=\mathcal{C}^1(x)+\int_0^\infty e^{-t\tau} e^{s\tau}\mathcal{M}(x) d\beta(\tau)
=\mathcal{C}^1(x)+\int_0^\infty e^{-t(\tau \mathcal{I})} [e^{s\tau}\mathcal{M}(x)] d\beta(\tau)
\end{aligned}
$$
has the form of \eqref{C_sum_exp} with $\mathcal{A}^\tau(x)=\tau \mathcal{I}$ and $\mathcal{C}_0^\tau(x,s)=e^{s\tau}\mathcal{M}(x)$.
For each $\tau\in(0,\infty)$, $\mathcal{A}^\tau$ and $\mathcal{C}_0^\tau$ certainly satisfy the conditions given in Remark \ref{cond_suff_F}.
By Corollary \ref{exponential_kernel}, this model yields finite propagation speed with maximum propagation speed not exceeding
$$
c=\sup_{x\in \bar B_{R} (x_0)}\left[\frac{\lambda_{max}^{\mathcal{C}^1}(x)+\lambda_{max}^{\mathcal{M}}(x)}{\rho(x)}\right]^{\frac 1 2}.
$$
This estimate of $c$ follows from the facts that $e^{-t \mathcal{A}^\tau(x)} \mathcal{C}_0^\tau(x,t)=\mathcal{M}(x)$ is independent of $\tau$ and $\int_0^\infty d\beta(\tau)=E_\alpha[0^-]=1$. The model \eqref{frac_zener} is called the fractional Zener model, because it establishes
a fractional differential relation between the stress and strain:
$$
\left[1+a\frac{d^\alpha}{dt^\alpha}\right]\sigma(x,t)=\left[\mathcal{C}^1(x)+
a\left(\mathcal{M}(x)+\mathcal{C}^1(x)\right)\frac{d^\alpha}{dt^\alpha}\right]\epsilon(x,t).
$$
Here $d^\alpha/dt^\alpha$ denotes the Caputo fractional derivative defined by
$$
\frac{d^\alpha f}{dt^\alpha} (t):=\frac{1}{\Gamma(1-\alpha)}\int_0^t \frac{\dot{f}(s)}{(t-s)^\alpha}ds.
$$
In \cite{M}, one may find more information on fractional calculus including the fractional Zener model.

\section{Conclusion and Discussion}
In this paper we started with work density and introduced a new type of decomposition called work density decomposition.  We apply this decomposition to a very general linear, anisotropic, viscoelastic, integro-differential model with variable coefficients.  Then using energy arguments including the space-time divergence theorem, we establish finite propagation speed results.   These results include a superposition principle that enables a possibly easier way to establish that the hypotheses of our theorem are satisfied.  The advantage here over plane wave arguments
(see e.g.\cite{M}) is that the anisotropic viscoelastic parameters can be variable. In the variable coefficient case we obtain an upper bound for
the wave propagation speed.

 While we primarily select a specific work density decomposition we have noted that this decomposition
 is not unique.  The reason for this is that the decomposition includes strain energy density and this quantity is not uniquely defined in the viscoelastic case.  This means that in the application of the decomposition to establish finite propagation speed results, the
 strain energy density must be chosen appropriately for any specific model. To illustrate this point, we provide some examples of practical models for which a different choice for strain energy density is needed to establish finite propagation speed.

\section*{Acknowledgments}
The first author, Joyce McLaughlin, was partially supported for this work by ONR N00014-13-1-0388 funding.

\end{document}